\documentclass[11pt]{article}

\usepackage{amssymb,amsmath,amsthm}
\usepackage{amsfonts}
\usepackage{latexsym}
\usepackage{enumerate}
\usepackage{wasysym}
\usepackage{fancyhdr}
\usepackage{lipsum}
\usepackage{epsfig}
\usepackage{breqn}
\usepackage{footnpag}
\usepackage{rotating}
\usepackage{amsfonts}
\usepackage{setspace}
\usepackage{fullpage}
\usepackage{enumitem}
\usepackage{bbold} 
\usepackage{comment}
\usepackage{pgf,tikz}
\usepackage{hyperref}
\usepackage{verbatim}
\usepackage{color}
\usepackage{mathrsfs}
\usepackage{soul}
\usetikzlibrary{arrows}
\usepackage{array}

\newtheoremstyle{dotless}{}{}{\itshape}{}{\bfseries}{}{ }{}

  \theoremstyle{dotless}

\newtheorem{theorem}{Theorem}[section] 

\newtheorem{conjecture}[theorem]{Conjecture}
\newtheorem{lemma}[theorem]{Lemma}

\newtheorem{observation}[theorem]{Observation}

\newtheorem{construction}[theorem]{Construction}

\DeclareMathOperator{\sat}{sat}

\title{Linearity of Saturation for Berge Hypergraphs}
\author{Sean English \thanks{Department of Mathematics, 
		Western Michigan University, Kalamazoo MI 49008-5240 USA}\and D\'aniel Gerbner \thanks{Hungarian Academy of Sciences, Alfr\'ed R\'enyi Institute of Mathematics, P.O.B. 127, Budapest H-1364, Hungary. Research supported by the J\'anos Bolyai Research Fellowship of the Hungarian Academy of Sciences and by the National Research, Development and Innovation Office -- NKFIH, grant K 116769.} \and  Abhishek Methuku \thanks{Department of Mathematics, Central European University, Budapest, N\'ador u. 9, 1051 Hungary. Research supported by National Research, Development and Innovation Office -- NKFIH, grant K 116769.} \and Michael Tait\thanks{Department of Mathematical Sciences, Carnegie Mellon University, Pittsburgh, PA, 15213, USA. Supported in part by NSF grant DMS-1606350.}}

\begin{document}
\maketitle

\begin{abstract}
For a graph $F$, we say a hypergraph $H$ is Berge-$F$ if it can be obtained from $F$ be replacing each edge of $F$ with a hyperedge containing it. We say a hypergraph is Berge-$F$-saturated if it does not contain a Berge-$F$, but adding any hyperedge creates a copy of Berge-$F$. The $k$-uniform saturation number of Berge-$F$, $\sat_k(n,\text{Berge-}F)$ is the fewest number of edges in a Berge-$F$-saturated $k$-uniform hypergraph on $n$ vertices. We show that $\sat_k(n,\text{Berge-}F) = O(n)$ for all graphs $F$ and uniformities $3\leq k\leq 5$, partially answering a conjecture of English, Gordon, Graber, Methuku, and Sullivan. We also extend this conjecture to Berge copies of hypergraphs.

\end{abstract}

\section{Introduction}

Given a family of graphs $\mathcal{F}$, a graph $G$ is {\em $\mathcal{F}$-saturated} if it does not contain any $F\in \mathcal{F}$ as a subgraph, but the addition of any edge creates a copy of some $F\in \mathcal{F}$. Thus, the maximum number of edges in an $\mathcal{F}$-saturated graph is the Tur\'an number for $\mathcal{F}$, denoted by $\mathrm{ex}(n, \mathcal{F})$. The study of Tur\'an numbers for various families of graphs is a cornerstone of extremal combinatorics, c.f. \cite{K, S} for surveys. On the other end of the spectrum, we define the {\em saturation number} of $\mathcal{F}$ to be the minimum number of edges in an $\mathcal{F}$-saturated graph and denote this quantity by $\mathrm{sat}(n, \mathcal{F})$. Saturation numbers were first studied by Erd\H{o}s, Hajnal, and Moon \cite{EHM} and since then have been researched extensively. K\'aszonyi and Tuza \cite{KT} showed that saturation numbers are always linear. That is, for any finite family $\mathcal{F}$ of graphs, there is a constant $C$ such that $\mathrm{sat}(n, \mathcal{F}) \leq Cn$. For more results on graph saturation, we refer the reader to the survey \cite{FFS}. 

Graph saturation has been generalized in several natural ways, including studying other host graphs besides the complete graph \cite{FJPW, KS}, adding edge colors \cite{BFVW, HT}, unique and weak saturation \cite{BS, B}, and the study of the saturation spectrum \cite{Thomas2}. In this paper, we are interested in considering saturation numbers of hypergraphs.

Given a family of $k$-uniform hypergraphs $\mathcal{F}$ and a $k$-uniform hypergraph ${H}$, we say that ${H}$ is $\mathcal{F}$-saturated if ${H}$ is $\mathcal{F}$-free but the addition of any $k$-edge creates a copy of some hypergraph in $\mathcal{F}$. We denote the minimum number of hyperedges in an $\mathcal{F}$-saturated graph by $\mathrm{sat}_k(n, \mathcal{F})$. Complementing the result of K\'aszonyi and Tuza, Pikhurko \cite{P} showed that for any finite family $\mathcal{F}$ of $k$-uniform hypergraphs, one has $\mathrm{sat}_k(n, \mathcal{F}) = O(n^{k-1})$.


Extending theorems in extremal graph theory to hypergraphs is a notoriously difficult problem in general. However, recent attempts to put graph structure on a hypergraph extremal problem have been successful in making problems tractable. Given a graph $F$ and a hypergraph ${H}$ on the same vertex set, we say that ${H}$ is {\em Berge-$F$} if there is a bijection $\phi: E(F) \to E({H})$ such that $e\subseteq \phi(e)$ for all $e\in E(F)$. That is, each edge of $F$ can be expanded to a unique hyperedge of ${H}$, or alternatively each hyperedge of ${H}$ may be shrunk down to a unique edge in $F$. Note that this generalizes the well-studied concept of a Berge cycle and a Berge path. Also note that for a graph $F$ there are in general many non-isomorphic hypergraphs which are Berge-$F$. We denote by $\mathcal{B}_k(F)$ the family of all $k$-uniform hypergraphs which are Berge-$F$.  

Given a graph $F$ the $k$-uniform {\em expansion} of $F$ is the hypergraph $F^+$ obtained from $F$ by enlarging each edge in $F$ with $k-2$ new vertices where distinct edges are enlarged by distinct vertices. Note that $F^+$ is a Berge-$F$. The study of Tur\'an numbers of $F^+$ (c.f. \cite{KMV1, KMV2, KMV3, M, P2}) or Berge-$F$ (c.f. \cite{GP, GKL, GL, GMT, LV, T, GMV}) for various graphs $F$ has received quite a lot of attention recently. 

In this paper, we study saturation numbers for Berge-$F$. For convenience, we will let $\sat_k(n,\text{Berge-}F)=\sat_k(n,\mathcal{B}_k(F))$. Bounds on the saturation numbers for many common classes of graphs were given in \cite{GRWCsat}. In addition, the authors conjectured that these numbers will always grow linearly in $n$, regardless of the uniformity or the graph $F$. Our main theorem confirms this conjecture for uniformities $k\in \{3,4,5\}$. This suggests that the saturation problem for the family Berge-$F$ behaves more like a family of graphs than one might expect from Pikhurko's general upper bound.

\begin{theorem}\label{theorem main}
	For any graph $F$ and any $k$ with $3\leq k\leq 5$, we have
	\[
	\sat_k(n,\text{Berge-}F)=O(n).
	\]
\end{theorem}

\subsection{Definitions, notation, and organization}\label{section notation}

Here we provide some basic definitions and notation that will be used throughout the paper. Given a graph $G$ on vertex set $V$, and a set $S \subset V$ of vertices, let $G[S]$ denote the subgraph of $G$ induced by $S$, and let $G - S$ denote the subgraph of $G$ induced by $V \setminus S$.  For convenience, we will call a connected component of a graph as simply a ``component".

We will assume that any forbidden graph $F$ does not have any isolated vertices. We do not lose anything by doing so, because if $F$ did have isolated vertices, all that might change is that a Berge-$F$ may also necessarily contain some isolated vertices, but since we are concerned with the saturation function for large $n$, any constant number of isolated vertices in a Berge-$F$ will not affect which hypergraphs on $n$ vertices are $\mathcal{B}_k(F)$-saturated.

For a graph $G$, let $\beta(G)$ denote the vertex cover number of $G$, i.e., the size of a smallest vertex set $A\subset V(G)$ such that every edge $e\in E(G)$ is incident with a vertex of $A$. A \emph{vertex feedback set} is a set of vertices $S$ such that $G-S$ is acyclic. The cardinality of a smallest vertex feedback set is the feedback number of $G$, denoted by $f(G)$. 

A $k$-uniform hypergraph $H$ on $n$ vertices is called $d$-degenerate if there is an ordering of the vertex set $V(H)=\{v_1,\dots,v_n\}$ such that $v_\ell$ has degree less than or equal to $d$ in the subhypergraph of $H$ induced by $V(H) \setminus \{v_1,\dots,v_{\ell-1}\}$. 
Note that if $H$ is $d$-degenerate, then $|E(H)|\leq dn$.

Suppose a hypergraph $H$ is Berge-$F$. If we embed $F$ on the vertex set of $H$ in such a way that there exists a bijection $\phi:E(F)\to E(H)$ such that $e\in \phi(e)$ for each $e\in E(F)$, then we will say that $F$ \emph{witnesses} $H$ and we call the set of vertices of $F$ in this embedding, the {\em core vertices}. By the definition of Berge-$F$, there is always at least one such embedding.

\vspace{2mm}

\textbf{Organization of the paper:} In Section \ref{section linearity}, we prove our main theorem, Theorem \ref{theorem main}. In Section \ref{section conjecture} we formulate a general conjecture about how the saturation function should behave for families of Berge hypergraphs. We show that this conjecture would be best possible if true.


\section{Proof of Theorem \ref{theorem main}}\label{section linearity}

\subsection{Proof Sketch}
The main idea of the proof involves the use of Construction \ref{construction HnaFS}, a hypergraph with few edges that is constructed based on a feedback set of the forbidden graph $F$. While we do not usually expect the construction to be Berge-$F$-saturated, we show that if the construction does not contain Berge-$F$, we can greedily add at most a linear number of edges to create a Berge-$F$-saturated hypergraph. When using this construction, some issues arise in certain cases if the forbidden graph $F$ is not connected. As such, we use two different results, one where the construction does not contain a Berge-$F$, and one where the construction does not contain a Berge copy of any connected component of $F$. 

With this in hand, the rest of the proof involves strategic use of this construction, depending on the vertex cover number, $\beta(F)$ and the vertex feedback number, $f(F)$. Most of the work is done by Theorem \ref{theorem beta>k}, which deals with any graph $F$ such that $\beta(F) >k$. From here, to show linearity for $k\in\{3,4,5\}$, we need only deal with a few cases depending on $f(F)$ and $\beta(F)$ when Theorem \ref{theorem beta>k} does not apply. 
\subsection{Proof}

Our main tool involves the following construction on a vertex set $V$ of size $n$ (where $n$ is sufficiently large):

\begin{construction}\label{construction HnaFS}
	Let $G$ be a graph and let $S$ be a vertex feedback set of $G$, and let $|E(G[S])|=\ell$. Let the vertices of $V$ be partitioned into three sets $V=V_1\cup V_2\cup V_3$, where $|V_1|=f(G)$ (recall that $f(G)$ denotes the vertex feedback number of $G$), and $|V_2|=(k-2)\ell$. We first add $\ell$ hyperedges between $V_1$ and $V_2$ to create a Berge-$G[S]$ with core vertices in $V_1$: To do this, arbitrarily label the vertices in $V_1$ with labels from the vertex set of $G[S]$, and then for each edge $uv$ of $G[S]$, add a $k$-edge that consists of the vertices labeled $u$ and $v$ in $V_1$, and $k-2$ vertices in $V_2$ in such a way that after all the $\ell$ edges are added, each vertex in $V_2$ has degree $1$ (i.e., for each edge $uv$, we choose some $k-2$ vertices in $V_2$ that were not chosen before, and add the hyperedge consisting of $u, v$ and these vertices).
	
	Now, choose some integer $1\leq a\leq k-1$ 
	such that $|S| \geq k-a$. If $a$ does not divide $|V_3|$, arbitrarily choose $(|V_3|\mod a)$ vertices, and remove them to form the set $V_3'\subset V_3$, with $|V_3'|=ra$ for some $r\in\mathbb{Z}$. Partition $V_3'$ into $r$ sets of size $a$, and let $\mathcal{M}$ be the collection of these $a$-sets. For each $a$-set $A$ in $\mathcal{M}$, add all the $\binom{|V_1|}{k-a}$ hyperedges that contain $A$ and some $k-a$ vertices from $V_1$. Call this hypergraph construction $H_k(n,a,G,S)$.
\end{construction}

We use this construction to show linearity of the saturation function in many cases as follows. Depending on the situation, sometimes the graph $G$ in Construction \ref{construction HnaFS} will be the entire forbidden graph $F$, and other times the graph $G$ will only be a collection of connected components of $F$. Due to this, in the following lemmas and theorems, we will let $F^*$ denote the subgraph of $F$ that we use in Construction \ref{construction HnaFS}. 

\begin{lemma}\label{lemma HnaFS construction}
	Let $F$ be a graph, and let $F^*$ be a subgraph of $F$ made up of a collection of components. Let $S$ be a vertex feedback set of $F^*$ and let $1\leq a\leq k-1$ be such that if $|S|\neq 0$, then $|S| > k-a$. Let $H^*$ be a Berge-$(F-F^*)$ and let $z=|V(H^*)|$. If the disjoint union $H_k(n-z,a,F^*,S)\cup H^*$ does not contain a Berge-$F$ (in the case when $F=F^*$, we need only that $H_k(n,a,F,S)$ does not contain a Berge-$F$), then $\sat_k(n,\text{Berge-}F)=O(n)$.
\end{lemma}

\begin{proof}
	Assume $H_k(n-z,a,F^*,S)\cup H^*$ is Berge-$F$-free. Let $s=|V(F^*)|-f(F^*)$. First we show that for each choice of $s$ vertices such that each vertex is from a distinct member of $\mathcal{M}$, $H_k(n-z,a,F^*,S)$ contains a Berge copy of $F^*[S]\vee \overline{K_s}$ where the core vertices corresponding to $\overline{K_s}$ are the $s$ chosen vertices, and the core vertices corresponding to $F^*[S]$ are in $V_1$. Indeed, we can use the $\ell$ edges added between $V_1$ and $V_2$ to create the Berge-$F^*[S]$. To create the Berge complete bipartite graph from $S$ to the $s$ vertices, note that if $a=k-1$ then there is a unique way to do this. If $a<k-1$, for each vertex $u\in V_3$, note that $u$ is in a distinct edge with each $(k-a)$-subset of $S$, and so we can easily build a Berge-$K_{1,|S|}$ with center at $u$ and the other core vertices in $|S|$ (for example, if we consider a $(k-a)$-uniform complete graph on vertex set $S$, this clearly contains a tight cycle, and so the edges corresponding to the tight cycle will form the desired Berge star). Note that since our $s$ chosen vertices are from distinct sets in $\mathcal{M}$, we can build a disjoint Berge-$K_{1,|S|}$ for each chosen vertex, giving us the desired complete bipartite graph.
	
	 Let us add edges to $H_k(n-z,a,F^*,S)\cup H^*$ arbitrarily until it is Berge-$F$-saturated. Let $H$ be the resulting hypergraph and let $H'$ be the hypergraph containing all of the added edges. $H'$ cannot contain a Berge copy of the forest $T=F^*-S$ such that each core vertex of the Berge-$T$ is contained in a different element of $\mathcal{M}$, since otherwise $H$ would contain a Berge-$F$. Let
	\[
	d=\binom{z+|V_1|+|V_2|+|V_3|-|V_3'|+(|V(T)|-2)a+a-1}{k-1}
	\]
	We claim that $H'$ is $d$-degenerate. Indeed, let us assume we have removed $b$ vertices $\{v_1,\dots,v_b\}$ from $H'$ such that each time we removed one, it had degree less than or equal to $d$. If the minimum degree of $H'-\{v_1,\dots,v_b\}$ is greater than $d$, we can build a Berge-$T$ such that each core vertex is in a different set of $\mathcal{M}$. To see this, let us assume to the contrary that this is not true, and let $T'$ be the largest subtree of $T$ such that $H'-\{v_1,\dots,v_b\}$ contains a Berge-$T'$ with each core vertex in a different set of $\mathcal{M}$. We will extend this Berge-$T'$ by one edge. Indeed, let $u$ be a core vertex corresponding to a vertex of $T'$ that has lower degree than its counterpart in $T$. Let $A_1,\dots,A_{|V(T')|}\in\mathcal{M}$ be the sets containing the core vertices of $T'$. We wish to find an edge containing $u$ that also contains a vertex $x$ that is not in $V(H^*)$, $V_1$, $V_2$, $V_3\setminus V_3'$, and not in any $A_i$, $1\leq i\leq |V(T')|$. There are no more than $(|V(T)|-2)a+a-1$ other vertices in $\bigcup_{i=1}^{|V(T')|}A_i$, so by our choice of $d$, we know there must be an edge containing a vertex avoiding all the desired sets. This vertex $x$ must be in some new set $A\in \mathcal{M}$, and so we can find a Berge tree larger than $T'$ with the desired properties, contradicting the maximality of $T'$.
	
	Thus, $H'$ is $d$-degenerate. Since $d$ depends only on $F$, $k$ and $a$ ($z\leq k|E(F)|$, $|V_1|=|S|\leq|V(F)|$, $|V_2|=(k-2)\ell$, and $|V_3|-|V_3'|<a$), this gives us that $H'$ has at most $dn=O(n)$ edges. As $H_k(n-z,a,F^*,S)$ also has linearly many edges (all but the $f(F^*)$ vertices in $V_1$ have bounded degree), and $H^*$ has finitely many edges, we have that the Berge-$F$-saturated graph $H$ has linearly many edges.
\end{proof}

We will often choose $F^*$ in the preceding lemma to be a single component. When this is the case, there is an easier condition to guarantee linearity.

\begin{lemma}\label{lemma HnaFS construction 1 component}
	let $F$ be a graph with components $F_1,\dots,F_c$, and let $F^*\in\{F_1,\dots,F_c\}$. For any choice of $S$, $a$ and $H^*$ satisfying the requirements of Lemma \ref{lemma HnaFS construction}, if $H_k(n-z,a,F^*,S)$ is Berge-$F_i$-free for $1\leq i\leq c$, then $\sat_k(n,\text{Berge-}F)=O(n)$.
\end{lemma}

\begin{proof}
	Note that the disjoint union $H_k(n-z,a,F^*,S)\cup H^*$ is Berge-$F$-free since there is no Berge-$F_i$ in $H_k(n-z,a,F^*,S)$, and $H^*$ has too few edges to contain a Berge-$F$. Thus, we are done by Lemma \ref{lemma HnaFS construction}.
\end{proof}

Now, we will show how to use Lemmas \ref{lemma HnaFS construction} and \ref{lemma HnaFS construction 1 component} to establish linearity for small uniformities. First, we make an easy, but important observation.

\begin{observation}\label{observation f=0}
	If $F$ is a graph with a component $F^*$ such that $f(F^*)=0$ (i.e. $F^*$ is a tree), $H_k(n,a,F^*,\emptyset)$ is empty, so Lemma \ref{lemma HnaFS construction 1 component} gives us that the saturation number is linear for any graph with an acyclic component.
\end{observation}

We will take care of most forbidden graphs $F$ by showing that if the vertex cover number of $F$ is not too small, then we are done.

\begin{theorem}\label{theorem beta>k}
	If $F$ is a graph with $\beta:=\beta(F)\geq k+1$, then
	\[
	\mathrm{sat}_k(n,\text{Berge-}F) = O(n).
	\]
\end{theorem}

\begin{proof}[Proof of Theorem \ref{theorem beta>k}]
	By Observation \ref{observation f=0}, we may assume every component of $F$ contains a cycle. Let $C$ be a minimum vertex cover of $F$. Let $S\subset C$ with $|S| = \beta-1>k-1$. Note that $S$ is a feedback set since the remaining edges in $F-S$ form a star centered at the vertex in $C\setminus S$. We will show that $H=H_k(n,1,F,S)$ is Berge-$F$-free, which will suffice to complete the proof by Lemma \ref{lemma HnaFS construction} (applied with $F=F^*$).
	
	Let $G$ be a graph such that Berge-$G$ is in $H$, and embed $G$ on the same vertex as $H$ such that $G$ witnesses Berge-$G$. Note that every edge of $G$ is either incident with a vertex in $V_1$ or is an isolated edge contained in $V_2$. Thus, if $G$ does not contain any acyclic components, $S\cap V(G)$ is a vertex cover of $G$, and thus $\beta(G)\leq|S|<\beta(F)$, so $H$ must be Berge-$F$-free, and we are done.
\end{proof}

By Theorem \ref{theorem beta>k}, in order to prove Theorem \ref{theorem main} we need only consider the case where $\beta(F) \leq k$. Note that for any graph $F$ we have $f(F) < \beta(F)$. In general, Theorem \ref{theorem beta>k} goes a long way towards proving the conjecture from \cite{GRWCsat} that the saturation number for Berge-$F$ is linear for any $k$ and $F$. Indeed, it can easily be shown that this implies that for any fixed $k$, the saturation number of Berge-$F$ grow linearly for almost all graphs $F$. However, the case analysis for general $k$ for when $\beta(F) \leq k$ and $f(F) \leq k-1$ becomes intractable for large $k$, and this is why we were only able to prove Theorem \ref{theorem main} for $k\in \{3,4,5\}$. In order to prove the theorem for general uniformities using our approach, some new insight would be needed on how to deal with many different feedback numbers or vertex cover numbers at once. In order to prove Theorem \ref{theorem main}, we only need to consider a subset of the cases where $\beta(F) \leq 5$ and $f(F) \leq 4$. We handle these in a few separate ways, which we summarize in the following table.

\begin{center}
\begin{tabular}{|c|c|c|c|}
\hline
 $k$&$\beta(F)$&$f(F)$&Place completed\\
\hline
 arbitrary&arbitrary&$0$&Observation \ref{observation f=0}\\
\hline
 arbitrary&arbitrary&$1$&Lemma \ref{lemma f=1}\\
\hline 
  arbitrary&arbitrary&$2$&Lemma \ref{lemma f=2}\\
\hline
 arbitrary&$4$, $5$&$3$&Lemma \ref{lemma f=3}\\
 \hline
 $5$&$5$&$4$&Lemma \ref{lemma f=4}\\
 \hline

\end{tabular}
\end{center}

We now go through the varying cases outlined in the preceding table. Note, the following lemmas are constructed in a way to rule out not just a graph with a certain feedback number, but a graph containing any component with a certain feedback number. In this way, each lemma can be used to simplify the work of the following lemmas.

\begin{lemma}\label{lemma f=1}
	Let $F$ be a graph and let $F^*$ be a component of $F$ such that $f(F^*)=1$. Then $\sat_k(n,\text{Berge-}F)=O(n)$.
\end{lemma}

\begin{proof}
	By Observation \ref{observation f=0}, we can assume that every component of $F$ has vertex feedback number at least $1$. 
	Let $z$ be the number of vertices in some Berge-$(F-F^*)$ (or $z=0$ if $F=F^*$). Let $v$ be a vertex of $F^*$ whose removal leaves a forest. Note that there is only one vertex in $H_k(n-z,k-1,F^*,\{v\})$ of degree $\geq 1$, so this graph cannot contain a Berge cycle. Since every component of $F$ contains a cycle, $H_k(n-z,k-1,F^*,\{v\})$ contains no Berge copy of any component of $F$, and so by Lemma \ref{lemma HnaFS construction 1 component}, we are done.
\end{proof}

\begin{lemma}\label{lemma f=2}
	Let $F$ be a graph with a component $F^*$ with $f(F^*)=2$. Then $\sat_k(n,\text{Berge-}F)=O(n).$
\end{lemma}

\begin{proof}
	By Observation \ref{observation f=0} and Lemma \ref{lemma f=1}, we may assume that every component of $F$ has vertex feedback number at least $2$. Let $z$ be the number of vertices in some Berge-$(F-F^*)$ (or $z=0$ if $F=F^*$). Let $\{u,v\}$ be a vertex feedback set of $F^*$. Consider $H=H_k(n-z,k-1,F^*,\{u,v\})$. We will show $H$ does not contain a Berge-$G$ for any graph $G$ with $f(G)\geq 2$. Assume to the contrary that $H$ does contain Berge-$G$, and embed $G$ such that $G$ witnesses Berge-$G$.
	
	Note that any Berge cycle in $H$ must use both vertices in $V_1$ as core vertices since the only pairs of adjacent vertices in $H$ aside from the vertices in $V_1$ share at most two edges among them. Then $G-v$ is acyclic, which contradicts $f(G)\geq 2$. Thus, $H$ is Berge-$G$-free, and thus contains no Berge copy of any components of $F$, so by Lemma \ref{lemma HnaFS construction 1 component}, we are done.
\end{proof}

\begin{lemma}\label{lemma f=3}
	Let $F$ be a graph with $\beta(F)\leq 5$ that contains a component $F^*$ with $f(F^*)=3$. Then $\sat_k(n,\text{Berge-}F)=O(n)$.
\end{lemma}

\begin{proof}
	By Observation \ref{observation f=0}, and Lemmas \ref{lemma f=1} and \ref{lemma f=2}, we can assume every component of $F$ has vertex feedback number at least $3$. Let $z$ be the number of vertices in some Berge-$(F-F^*)$ (or $z=0$ if $F=F^*$). Let $S$ be a minimum vertex feedback set of $F^*$, Consider $H=H_k(n-z,k-2,F^*,S)$. We will show $H$ does not contain a Berge-$G$ for any $G$ with $f(G)\geq 3$ and $\beta(G)\leq 5$. Assume to the contrary that $H$ does, and embed $G$ such that $G$ witnesses Berge-$G$. Note that any cycle in $G$ must use a vertex in $V_1$ since otherwise the cycle would be contained in one of the sets $A\in\mathcal{M}$, and only three edges of $H$ are incident with vertices in $A$, and by the connectedness of $G$, one of these edges in $G$ must leave the set $A$. This implies that $V_1$ is a minimum feedback set. Furthermore, for $v\in V_1$, $v$ must be in a cycle that does not involve vertices in $V_1\setminus\{v\}$ (otherwise there would be a vertex feedback set of size $2$). This cycle must be a triangle with one edge in some set $A\in\mathcal{M}$. Furthermore, this cycle must use all the edges incident with vertices in $A$, so there exists three disjoint triangles, one for each vertex in $V_1$ in $F^*$. Since $\beta(3K_3)=6>\beta(G)$, we reach a contradiction. Thus, $H$ is Berge-$G$-free, and so $H$ does not contain a Berge copy of any components of $F$ so by Lemma \ref{lemma HnaFS construction 1 component}, we are done.
\end{proof}

The preceding work applies to all uniformities $k\geq 3$, and is sufficient to show that the saturation numbers $\sat_k(n,\text{Berge-}F)$ grow at most linearly for $k=3,4$ and all $F$. When $k=5$, we still need to deal with the case when $f(F)=4$ though. The next lemma is specific to $k=5$, and completes the proof of the main theorem.

\begin{lemma}\label{lemma f=4}
	Let $F$ be a graph with $\beta(F)=5$ and $f(F)=4$. Then $\sat_5(n,\text{Berge-}F)=O(n)$.
\end{lemma}

\begin{proof}
	By Observation \ref{observation f=0} and Lemmas \ref{lemma f=1}, \ref{lemma f=2} and \ref{lemma f=3}, we can assume $F$ does not contain any component with feedback number less than $4$. Since $f(F)=4$, this implies that $F$ is connected. Let $S$ be a minimum vertex feedback set of $F$. We will show that $H=H_5(n,2,F,S)$ does not contain a Berge-$F$. Along with Lemma \ref{lemma HnaFS construction 1 component} (applied with $F^*=F$), this will complete the proof.
	
	Assume to the contrary that $H$ does contain a Berge-$F$. Embed $F$ in the vertex set of $H$ such that $F$ witnesses Berge-$F$. Since $a=2$ and since vertices in $V_2$ have degree $1$ in $H$, we have that $F-(V(F)\cap V_1)$ is a matching. Since $|V_1|=f(F)$, this implies that $V_1\subset V(F)$, and that $V_1$ is a minimum feedback set of $F$.
	
	Let $v\in V_1$. Then $F-(V_1\setminus\{v\})$ must contain a cycle that goes through $v$. By the structure of $H$, this cycle must be a triangle with vertices $v$, $x$ and $y$, where $\{x,y\}\in\mathcal{M}$. $v$ was chosen arbitrarily from $V_1$, so every vertex in $V_1$ must be in a triangle with a pair from $\mathcal{M}$. We claim that $F$ must contain three disjoint triangles. Indeed, let $\{x,y\}\in\mathcal{M}$. Note that only $4$ edges are incident with $x$ or $y$ in $H$, and so $x$ and $y$ are involved in at most $1$ triangle. Thus, for each vertex $v\in V_1$ to be in such a triangle, we need to use three pairs from $\mathcal{M}$, which gives us our three disjoint triangles. $\beta(3K_3)=6>\beta(F)$, giving us a contradiction. Thus, $H$ does not contain a Berge-$F$ and we are done.
\end{proof}

\section{Berge saturation for forbidden hypergraphs}\label{section conjecture}

let $F^{(r)}$ be an $r$-uniform hypergraph. Then for $k>r$, we say a $k$-uniform hypergraph $H$ is a $\text{Berge}_k\text{-}F^{(r)}$ if
there exists a bijection $\phi:E(F^{(r)})\to E(H)$ such that $e\subseteq \phi(e)$ for all $e\in E(F^{(r)})$. If the uniformity $k$ of the host hypergraph is clear from context, then we will refer to a $\text{Berge}_k\text{-}F^{(r)}$ as simply a Berge-$F^{(r)}$. Note that when $r=2$, this definition is consistent with the definition of Berge-$F$ in Section \ref{section notation}.

Here we present a conjecture that generalizes the conjecture in \cite{GRWCsat}.

\begin{conjecture}\label{conjecture higher uniform forbidden graph}
	Let $3\leq r<k$ be intergers and let $F^{(r)}$ be an $r$-uniform hypergraph. Then
	\[
	\sat_k(n,\text{Berge-}F^{(r)})=O(n^{r-1}).
	\]
\end{conjecture}

If Conjecture \ref{conjecture higher uniform forbidden graph} is true, then it is in some sense best possible. The tight path $P^{(r)}_\ell$ is an $r$-uniform hypergraph on $\ell$ vertices such that there exists an ordering of the vertex set $V(P^{(r)}_\ell)=\{v_1,\dots,v_\ell\}$ such that the edge set $E(P^{(r)}_\ell)$ consists of exactly the $\ell-r+1$ sets of $r$ consecutive vertices (in this ordering).  

The following theorems show the preceding conjecture is best possible, and also establishes the establishes the growth rates of both the saturation numbers and extremal numbers of Berge-$P^{(r)}_\ell$.

\begin{theorem}
	Let $3\leq r<k<\ell$ be integers and let $H$ be a $k$-uniform Berge-$P^{(r)}_\ell$-saturated hypergraph on $n$ vertices. Then
	\[
	|E(H)|=\Theta(n^{r-1}).
	\]
	
\end{theorem}

This will easily follow from the next two theorems. Indeed, we apply the next theorem with $t = r-1$ to obtain a lower bound of $\Omega(n^{r-1})$, and an upper bound of $O(n^{r-1})$ follows since $\sat_k(n, \text{Berge-}P_l^{(r)}) \le ex_k(n,\text{Berge-}P_l^{(r)})$.

\begin{theorem}
	Let $F$ be a hypergraph such that for every hyperedge in $F$ there is another one that shares at least $t$ vertices with it. Let $k\ge t$. Then $\sat_k(n, \text{Berge-}F)=\Omega(n^t)$.
\end{theorem}

\begin{proof}
	Let $H$ be a $k$-uniform Berge-$F$-free saturated hypergraph with $m$ hyperedges. Since $H$ is saturated, any non-edge $e\in E(\overline{H})$ must intersect some edge in at least $t$ vertices, since otherwise $e$ cannot be in a copy of $F$ in $H+e$. A fixed edge intersects at most
	\[
	d:=\sum_{a=1}^{k-t} \binom{k}{k-a}\binom{n-k}a=O(n^{k-t})
	\]
	non-edges in at least $t$ vertices. Thus the number of non-edges is given by
	\[
	\binom{n}k-m\leq md.
	\]
	Solving for $m$ in the preceding inequality gives
	\[
	m\geq \binom{n}k/(d-1)=\Omega(n^{t}).
	\]
	
\end{proof}

\begin{theorem}
$ex_k(n,\text{Berge-}P_l^{(r)})=O(n^{r-1})$.
\end{theorem}

\begin{proof}
Our proof is inspired by the reduction lemma of Gy\H ori and Lemons \cite{GL}. Let us consider a $k$-uniform $\text{Berge-}P_l^{(r)}$-free hypergraph $H$ and go through its hyperedges in an arbitrary order. For each hyperedge we pick a subset of it of size $r-1$. Among the $\binom{k}{r-1}$ subsets, we pick one which has been picked the least times earlier. We say that the multiplicity $m(A)$ of an $(r-1)$-set $A$ is the number of times it was picked during this process. If at the end of this algorithm every $(r-1)$-set has multiplicity less than $c=(\ell-r+1) \binom{\ell}{k}$, then there are at most $c\binom{n}{r-1}$ hyperedges in $H$ and we are done.

Therefore, we can assume there is an $(r-1)$-set $A_{r-1}=\{v_1,\dots,v_{r-1}\}$ with multiplicity at least $c$. It obtained multiplicity $c$ from a hyperedge $e_{r}$ where each $(r-1)$-set had multiplicity at least $c-1$ (at that point of the algorithm already). In particular there is a vertex $v_{r} \in e_{r}$ such that $A_{r}=\{v_2,\dots,v_r\}$ has multiplicity at least $c-1$. We will find $v_{r+1}, \dots, v_\ell$ similarly. More precisely, we will show by induction on $i$ with $r\leq i\leq\ell$ that after picking $e_i$ and $v_i$, we have that $A_i=\{v_{i-r+2}, \dots, v_i\}$ has multiplicity at least $(\ell-i)\binom{\ell}{k}$. The base step of the induction ($i=r$) follows from our assumption on $A_{r-1}$. Let us assume it holds for $i-1$ -- that is, the multiplicity of $A_{i-1}$ is at least $(\ell-i+1)\binom{\ell}{k}$. Less than $\binom{\ell}{k}>\binom{i-1}{k}$ of the hyperedges containing $A_{i-1}$ are contained in $\{v_1,\dots, v_{i-1}\}$. Thus one of the last $\binom{\ell}{k}$ hyperedges where we picked $A_{i-1}$ contains a new vertex, $v_i$. Let this hyperedge containing $v_i$ be $e_i$. During the algorithm, the addition of $e_i$ increased the multiplicity of $A_{i-1}$ to at least $(\ell-i)\binom{\ell}{k}+1$. This means that $A_i=\{v_{i-r+2},\dots,v_i\}$ must have had multiplicity at least $(\ell-i)\binom{\ell}{k}$ when $e_i$ was added, which is what we wanted to show.

In this way we can choose the edges $e_i$ to build a Berge tight path. Note that as long as $i\leq\ell$, $A_{i-1}$ has multiplicity at least $\binom{\ell}k>\binom{i-1}{k}$, and so a choice of $e_i$ is guaranteed to exist. Thus, the edges $e_{r},e_{r+1},\dots,e_{\ell}$ form a Berge-$P^{(r)}_\ell$.

\end{proof}

\bibliographystyle{plain}
\bibliography{bibv2}

\end{document}